\documentclass[11pt]{amsart}
\usepackage{amssymb,amscd,amsxtra,calc}
\usepackage{cmmib57}
\usepackage{graphicx}
\usepackage{amssymb,amsmath,amsfonts,mathrsfs,enumerate,xcolor}
\usepackage{amsthm}
\usepackage{soul}
\usepackage[citecolor=Navy]{hyperref}
\usepackage[ruled,vlined]{algorithm2e}

\usepackage{lipsum}
\usepackage{amsfonts}
\usepackage{graphicx}
\usepackage{epstopdf}
\usepackage{algorithmic}
\usepackage{float}
\restylefloat{table}
\usepackage{placeins}
\usepackage{array}
\ifpdf
  \DeclareGraphicsExtensions{.eps,.pdf,.png,.jpg}
\else
  \DeclareGraphicsExtensions{.eps}
\fi

\usepackage{subcaption}

\theoremstyle{plain}
    \newtheorem{thm}{Theorem}[section]

    \newtheorem{lemma}[thm]{Lemma}

    \newtheorem{theorem}[thm]{Theorem}

\theoremstyle{definition}

\theoremstyle{remark}

\newcommand{\authorfootnotes}{\renewcommand\thefootnote{\@fnsymbol\c@footnote}}

\usepackage{caption} 
 \title[New Q-Newton's method meets Backtracking line search]{New Q-Newton's method meets Backtracking line search: good convergence guarantee, saddle points avoidance, quadratic rate of convergence, and easy implementation}
 \author{Tuyen Trung Truong}
   \address{Department of Mathematics, University of Oslo, Blindern 0851 Oslo, Norway}
  \email{tuyentt@math.uio.no}
    \date{\today}
    \keywords{Backtracking line search, Convergence guarantee, Newton's method, Rate of convergence, Saddle points}
   \subjclass[2010]{}

\begin{document}
\maketitle
%{\centering\footnotesize To my daughter on her birthday occasion\par}

\begin{abstract}
In a recent joint work, the author has developed a modification of Newton's method, named New Q-Newton's method, which can avoid saddle points and has quadratic rate of convergence. While good theoretical convergence guarantee has not been established for this method, experiments on small scale problems show that the method works very competitively against other well known modifications of Newton's method such as Adaptive Cubic Regularization and BFGS, as well as first order methods such as Unbounded Two-way Backtracking Gradient Descent.

In this paper, we resolve the convergence guarantee issue by proposing a modification of New Q-Newton's method, named New Q-Newton's method Backtracking, which incorporates a more sophisticated use of hyperparameters and  a Backtracking line search. This new method has very good theoretical guarantees, which for a {\bf Morse function} yields the following (which is unknown for New Q-Newton's method): 

{\bf Theorem.} Let $f:\mathbb{R}^m\rightarrow \mathbb{R}$ be a Morse function, that is all its critical points have invertible Hessian. Then for a sequence $\{x_n\}$ constructed by New Q-Newton's method Backtracking from a random initial point $x_0$, we have the following two alternatives: 

i) $\lim _{n\rightarrow\infty}||x_n||=\infty$,

or 

ii) $\{x_n\}$ converges to a point $x_{\infty}$ which is a {\bf local minimum} of $f$, and the rate of convergence is {\bf quadratic}. 

Moreover, if $f$ has compact sublevels, then only case ii) happens. 

As far as we know, for Morse functions, this is the best theoretical guarantee for iterative optimization algorithms so far in the literature. A similar result was proven by the author for a modification of Backtracking Gradient Descent, for which the rate of convergence is only linear. 

New Q-Newton's method Backtracking can also be defined on Riemannian manifolds, and it can easily be implemented in Python. We have tested in experiments on small scale, with some further simplified versions of New Q-Newton's method Backtracking, and found that the new method significantly improve (either in computing resource, running time or convergence to a better point) the performance of New Q-Newton's method, and is competitive against other iterative methods (both first and second orders). Hence, this new method can be a good candidate for a second order iterative method to use in large scale optimisation such as in Deep Neural Networks.

\end{abstract}

\section{Introduction} Let $f:\mathbb{R}^m\rightarrow \mathbb{R}$ be a $C^2$ function, with gradient $\nabla f(x)$ and Hessian $\nabla ^2f(x)$.   Newton's method $x_{n+1}=x_n-(\nabla ^2f(x_n))^{-1}\nabla f(x_n)$ (if the Hessian is {\bf invertible}) is a very popular iterative optimization method. It seems that every month there is at least one paper about this subject appears on arXiv. One attractive feature of this method is that if it {\bf converges} then it usually converges very fast, with the rate of convergence being quadratic, which is generally faster than that of Gradient descent (GD) methods.  We recall that if $\{x_n\}\subset \mathbb{R}^m$  converges to $x_{\infty}$, and so that $||x_{n+1}-x_{\infty}||=O(||x_n-x_{\infty}||^{\epsilon})$, then $\epsilon$ is the rate of convergence of the given sequence. If $\epsilon =1$ then we have linear rate of convergence, while if $\epsilon=2$ then we have quadratic rate of convergence.  

However, there is no guarantee that Newton's method will converge, and it is problematic near points where the Hessian is not invertible. Moreover, it cannot avoid saddle points. Recall that a {\bf saddle point} is a point $x^*$ which is a non-degenerate critical point of $f$ (that is $\nabla f(x^*)=0$ and $\nabla ^2f(x^*)$ is invertible) so that the Hessian has at least one {\bf negative eigenvalue}. Saddle points are problematic in large scale optimization (such as those appearing in Deep Neural Networks, for which the dimensions could easily be millions or billions), see \cite{bray-dean, dauphin-pascanu-gulcehre-cho-ganguli-bengjo}. 

One way to improve the performance of Newton's method is to use a (Backtracking) line search, if the step direction of Newton's method $w_n=(\nabla ^2f(x_n))^{-1}\nabla f(x_n)$ is {\bf descent}, that is $<w_n,\nabla f(x_n)> >0$ (here $<.,.>$ is the usual inner product on $\mathbb{R}^m$). In this case one can choose a step size $\gamma >0$ for which a weak descent property $f(x_n-\gamma w_n)\leq  f(x_n)$ is satisfied, or the stronger Armijo's condition \cite{armijo} $f(x_n-\gamma w_n)-f(x_n)\leq -\gamma <w_n,\nabla f(x_n)>/2$ is satisfied. However, except if the cost function $f$ is convex, there is no guarantee that the step direction of Newton's method is a descent direction, and hence this idea cannot be directly applied. In current literature, there is one modification of Newton's method, namely Adaptive Cubic Regularization,  which has a descent property \cite{nesterov-polyak, cartis-etal}. While this method has quite good theoretical guarantees (though, under some quite restrictive assumptions and is not known if can avoid saddle points), its description is quite complicated and it depends on an auxilliary optimization subproblem at each step in the iteration process. This hinders an effective implementation of the method for actual computations. 

An overview of theoretical properties and a comparison of experimental performance of Adaptive Cubic Regularization (Python implementation in the GitHub link \cite{ARCGitHub}), against  Newton's method and some  modifications such as BFGS, Inertial Newton's method \cite{bolte-etal} and New Q-Newton's method (in the author's recent joint work \cite{truong-etal}), as well as a first order method Unbounded Two-way Backtracking GD \cite{truong-nguyen1},  is available in \cite{truong-etal}.  We refer interested readers to  see \cite{truong-etal} for the mentioned theoretical overview.  Some experimental results (on small scale problems) there will be reported again in the below in comparison with the new method to be developed in the current paper. 

The author's joint paper \cite{truong-etal} defined a new modification of Newton's method which is easy to implement, while can avoid saddle points and as fast as Newton's method. It is recalled in Algorithm \ref{table:alg}. Here we explain notations used in the description. Let $A:\mathbb{R}^m\rightarrow \mathbb{R}^m$ be an invertible {\bf symmetric} square matrix. In particular, it is diagonalisable.  Let $V^{+}$ be the vector space generated by eigenvectors of positive eigenvalues of $A$, and $V^{-}$ the vector space generated by eigenvectors of negative eigenvalues of $A$. Then $pr_{A,+}$ is the orthogonal projection from $\mathbb{R}^m$ to $V^+$, and  $pr_{A,-}$ is the orthogonal projection from $\mathbb{R}^m$ to $V^-$. As usual, $Id$ means the $m\times m$ identity matrix.

\medskip
{\color{blue}
 \begin{algorithm}[H]
\SetAlgoLined
\KwResult{Find a critical point of $f:\mathbb{R}^m\rightarrow \mathbb{R}$}
Given: $\{\delta_0,\delta_1,\ldots, \delta_{m}\}\subset \mathbb{R}$\ (chosen {\bf randomly}) and $\alpha >0$;
Initialization: $x_0\in \mathbb{R}^m$\;
 \For{$k=0,1,2\ldots$}{ 
    $j=0$\\
    \If{$\|\nabla f(x_k)\|\neq 0$}{
   \While{$\det(\nabla^2f(x_k)+\delta_j \|\nabla f(x_k)\|^{1+\alpha}Id)=0$}{$j=j+1$}}

$A_k:=\nabla^2f(x_k)+\delta_j \|\nabla f(x_k)\|^{1+\alpha}Id$\\
$v_k:=A_k^{-1}\nabla f(x_k)=pr_{A_k,+}(v_k)+pr_{A_k,-}(v_k)$\\
$w_k:=pr_{A_k,+}(v_k)-pr_{A_k,-}(v_k)$\\
$x_{k+1}:=x_k-w_k$
   }
  \caption{New Q-Newton's method} \label{table:alg}
\end{algorithm}
}
\medskip
 
 Experimentally, on small scale problems, New Q-Newton's method works quite competitive against the methods mentioned above, see \cite{truong-etal}.  However, while it can avoid saddle points, it does not have a descent property, and an open question in \cite{truong-etal} is whether it has good convergence guarantees. The current paper shows that this is the case, after one incorporates Backtracking line search into it. This is based on the observation that the step direction w constructed by New Q-Newton's method indeed is a {\bf descent direction}, see Lemma \ref{Lemma1}. We name the new method New Q-Newton's method, see Algorithm  \ref{table:alg0} below. Its use of hyperparameters is more sophisticated than that of New Q-Newton's method, and we need some notations.  For a square matrix $A$, we define: 
  
  $sp(A)=$ the maximum among $|\lambda |$'s, where $\lambda  $ runs in the set of eigenvalues of $A$, this is usually called the spectral radius in the Linear Algebra literature;
  
  and 
  
  $minsp(A)=$ the minimum among $|\lambda |$'s, where $\lambda  $ runs in the set of eigenvalues of $A$, this number is non-zero precisely when $A$ is invertible.

In Algorithm  \ref{table:alg0} below, note that by Lemma \ref{Lemma0}, the first While loop will terminate.  Also, the second  While loop terminates after a finite number of steps, since $w_k$ is a descent direction, see Lemma \ref{Lemma1}.   
 
\medskip
{\color{blue}
 \begin{algorithm}[H]
\SetAlgoLined
\KwResult{Find a critical point of $f:\mathbb{R}^m\rightarrow \mathbb{R}$}
Given: $\{\delta_0,\delta_1,\ldots, \delta_{m}\} \subset \mathbb{R}$\ (chosen {\bf randomly}, and $\alpha >0$;
Initialization: $x_0\in \mathbb{R}^m$\;
$\kappa:=\frac{1}{2}\min _{i\not=j}|\delta _i-\delta _j|$;\\
 \For{$k=0,1,2\ldots$}{ 
    $j=0$\\
  \If{$\|\nabla f(x_k)\|\neq 0$}{
   \While{$minsp(\nabla^2f(x_k)+\delta_j \|\nabla f(x_k)\|^{1+\alpha}Id)<\kappa  \|\nabla f(x_k)\|^{1+\alpha}$}{$j=j+1$}}
  
 $A_k:=\nabla^2f(x_k)+\delta_j \|\nabla f(x_k)\|^{1+\alpha}Id$\\
$v_k:=A_k^{-1}\nabla f(x_k)=pr_{A_k,+}(v_k)+pr_{A_k,-}(v_k)$\\
$w_k:=pr_{A_k,+}(v_k)-pr_{A_k,-}(v_k)$\\
$\widehat{w_k}:=w_k/\max \{1,||w_k||\}$\\
$\gamma :=1$\\
 \If{$\|\nabla f(x_k)\|\neq 0$}{
   \While{$f(x_k-\gamma \widehat{w_k})-f(x_k)>-\gamma <\widehat{w_k},\nabla f(x_k)>/2$}{$\gamma =\gamma /2$}}

$x_{k+1}:=x_k-\gamma \widehat{w_k}$
   }
  \caption{New Q-Newton's method Backtracking} \label{table:alg0}
\end{algorithm}
}
\medskip
 
There are two main differences between New Q-Newton's method  and New Q-Newton's method Backtracking. First, in the former we only need $\det(A_k)\not= 0$, while in the latter we need the eigenvalues of $A_k$ to be  "sufficiently large". Second, the former has no line search component, while the latter has. Note that the $\widehat{w_k}$ in Algorithm \ref{table:alg0} satisfies $||\widehat{w_k}||\leq 1$, and if $||\widehat{w_k}||<1$ then $\widehat{w_k}=w_k$. On the other hand, if one wants to keep closer to New Q-Newton's method, one can apply line search directly to $w_k$, and obtain a variant which will be named New Q-Newton's method Backtracking S, see Section \ref{Section2} for its description and theoretical guarantees. 
 
While New Q-Newton's method Backtracking is descriptive enough to be straight forwardly implemented in Python, in experiments we opt to some even simpler versions of it. Among many other things we will replace the first While loop of its by the first While loop of New Q-Newton's method.  We will use the  following two versions: New Q-Newton's method Backtracking V1 for which a weak descent property $f(x_n-\gamma w_n)\leq f(x_n)$ is incorporated (experiments show that it works very well, attains a similar point as V2 while being simpler than V2, and in certain cases can be much faster than V2), and the other is New Q-Newton's method Backtracking V2 for which Armijo's line search is incorporated. They are described in Algorithm \ref{table:alg1} and Algorithm \ref{table:alg2}. An explanation of why they should behave well, as observed in experiments, will be given in Section \ref{Section2}.  The S' version of V1 will be named V3, and that of V2 will be named V4, and they will also be detailed in Section \ref{Section2}, see Algorithms  \ref{table:alg3} and  \ref{table:alg4}.  

In the remaining of this section, we present the good theoretical guarantees of New Q-Newton's method Backtracking.  We recall that if $\{x_n\}\subset \mathbb{R}^m$ has a subsequence $\{x_{n_k}\}$ converging to a point $x_{\infty}$, then $x_{\infty}$ is a {\bf cluster point} of $\{x_n\}$. By definition, the sequence $\{x_n\}$ is convergent if it is bounded and has exactly one cluster point. 

\begin{theorem}
Let $f:\mathbb{R}^m\rightarrow \mathbb{R}$ be a $C^3$ function. Let $\{x_n\}$ be a sequence constructed by the New Q-Newton's method Backtracking. 

0) (Descent property) $f(x_{n+1})\leq f(x_n)$ for all n. 

1) If $x_{\infty}$ is a {\bf cluster point} of $\{x_n\}$, then $\nabla f(x_{\infty})=0$. That is, $x_{\infty}$ is a {\bf critical point} of $f$.

2) There is a set $\mathcal{A}\subset \mathbb{R}^m$ of Lebesgue measure $0$, so that if $x_0\notin \mathcal{A}$, and $x_n$ converges to $x_{\infty}$, then $x_{\infty}$ cannot be  a {\bf saddle point} of $f$. 

3) If $x_0\notin \mathcal{A}$ (as defined in part 2) and $\nabla ^2f(x_{\infty})$ is invertible, then $x_{\infty}$ is a {\bf local minimum} and the rate of convergence is {\bf quadratic}. 

4) More generally, if $\nabla ^2f(x_{\infty})$ is invertible (but no assumption on the randomness of $x_0$), then the rate of convergence is at least linear. 

5) If $x_{\infty}'$ is a non-degenerate local minimum of $f$, then for initial points $x_0'$ close enough to $x_{\infty}'$, the sequence $\{x_n'\}$  constructed by New Q-Newton's method will converge to $x_{\infty}'$. 

\label{Theorem1}\end{theorem}

There is a parallel result for New Q-Newton's method (see \cite{truong-etal}), where one obtains exactly the same statements for parts 2)-5), while in the corresponding part 1) for New Q-Newton's method one needs to assume in addition that the whole sequence $x_n $ {\bf converges} to $x_{\infty}$. When the cluster set of $\{x_n\}$ constructed by New Q-Newton's method has more than 1 point, theoretically it may happen that some of these cluster points are not {\bf critical} points of $f$, and this seems to be observed for the non-smooth function $f_3(x,y)=10(y-|x|)^2+|1-x|$ in the experiments in Section 3 (interestingly, one of these points is close to the global minimum of the function). The descent property part 0) is not true for New Q-Newton's method. There is also a parallel result for Backtracking GD and its modifications \cite{truong-nguyen1, truong-nguyen2}, with some differences: for Backtracking GD one can show additionally that either $\lim _{n\rightarrow\infty}||x_n||=\infty$ or $\lim _{n\rightarrow\infty}||x_{n+1}-x_n||=0$, while there is no assertion about quadratic rate of convergence in part 3).  

The next application concerns Morse functions. We recall that $f$ is Morse if every critical point of its is non-degenerate. That is, if $\nabla f(x^*)=0$ then $\nabla ^2f(x^*)$ is invertible. By transversality theory, Morse functions are dense in the space of functions. We recall that a sublevel of $f$ is a set of the type $\{x\in \mathbb{R}^m:~f(x)\leq a\}$, for a given $a\in \mathbb{R}$. The function $f$ has compact sublevels if all of its sublevels are compact sets. It is not known if New Q-Newton's method satisfies Theorem \ref{Theorem2}.  

\begin{theorem}
Let $f$ be a $C^3$ function, which is Morse. Let $x_0$ be a random initial point, and let $\{x_n\}$ be a sequence constructed by the New Q-Newton's method Backtracking. Then either  $\lim _{n\rightarrow\infty}||x_n||=\infty$, or $x_n$ converges to a {\bf local minimum}. In the latter case, the rate of convergence is {\bf quadratic}. Moreover, if $f$ has compact sublevels, then only the second alternative occurs. 
\label{Theorem2}\end{theorem}

As far as we know, for Morse functions, this is the best theoretical guarantee for iterative optimization algorithms so far in the literature. There is  a parallel result for a {\bf modification} of Backtracking GD \cite{truong, truong4, truongnew} (but it is an open question if the same holds for Backtracking Gradient Descent itself), where however there is no guarantee that the rate of convergence is quadratic. (Note that Standard GD can also avoid saddle points, see  \cite{lee-simchowitz-jordan-recht, panageas-piliouras}, but does not have good convergence guarantee if the function does not have globally Lipschitz continuous gradient. Therefore, Theorem \ref{Theorem2} is not available in general for sequences constructed by Standard GD.)

The remaining of this paper is organised as follows. In Section \ref{Section2} we prove Theorems \ref{Theorem1} and \ref{Theorem2}. There we also introduce some variants, and establish similar theoretical guarantees for them (in particular, the New Q-Newton's method Backtracking S version).  The section after presents some small scale experimental results (some are on non-smooth functions). These show that experimentally New Q-Newton's method Backtracking can significantly improve (either in computing resource, running time or convergence to a better point) New Q-Newton's method, and is competitive against other iterative methods (both first and second orders).  This gives more support that New Q-Newton's method Backtracking can be a good  second order iterative method candidate for large scale optimization such as in Deep Neural Networks. The last section states some  conclusions and generalisations (e.g. to Riemannian manifolds) and some ideas for future research directions.

{\bf Acknowledgments.} The author would like to thank Thu-Hien To for helping with some experiments. The author is partially supported by Young Research Talents grant 300814 from Research Council of Norway.

\section{Theoretical guarantees and Variants}\label{Section2} In this section we prove the main theoretical results in the introduction. Then we introduce some variants of New Q-Newton's method Backtracking, including those to be actually implemented for experiments. 

\subsection{Theoretical guarantees}

We start with auxilliary lemmas.

\begin{lemma} Let $\delta _0,\ldots ,\delta _m$ be distinct real numbers. Define
\begin{eqnarray*}
\kappa :=\frac{1}{2}\min _{i\not=j}|\delta _i-\delta _j|. 
\end{eqnarray*}
Let $A$ be an arbitrary symmetric $m\times m$ matrix, and $\epsilon $ an arbitrary real number. Then there is $j\in \{0,\ldots ,m\}$ so that $minsp(A+\delta _j\epsilon Id)\geq |\kappa \epsilon|$.

\label{Lemma0}\end{lemma}
 \begin{proof}
 Let $\lambda _1,\ldots ,\lambda _m$ be the eigenvalues of $A$. By definition
 \begin{eqnarray*}
 minsp(A+\delta _j\epsilon Id)=\min _{i=1,\ldots ,m}|\lambda _i+\delta _j\epsilon |. 
 \end{eqnarray*}
 The RHS can be interpreted as the minimum of the distances between the points $-\lambda _i$ and $\delta _j\epsilon$. The set $\{\delta _0\epsilon, \delta _1\epsilon, \ldots ,\delta _m\epsilon \}$ consists of $m+1$ distinct points on the real line, the distance between any 2 of them is $\geq 2|\kappa \epsilon|$.  Hence, by the pigeon hole principle, there is at least one $j$ so that the distance from $-\lambda _i$ to $\delta _j\epsilon $ is $\geq |\kappa \epsilon|$ for all $i=1,\ldots ,m$, which is what needed.  
 \end{proof}

The next lemma considers quantities such as $A_k,v_k,w_k,\ldots $ which appear in Algorithm \ref{table:alg0} for New Q-Newton's method Backtracking.  We introduce hence global notations for them for the ease of exposition. Fix a sequence $\delta _0,\delta _1,\ldots ,\delta _j,\delta _{j+1},\ldots $ of real numbers, where $\delta _j\in [h(j,m),h(j,m)+1]$. Fix also a number $\alpha >0$, and a $C^2$ function $f:\mathbb{R}^m\rightarrow \mathbb{R}$. For $x\in \mathbb{R}^m$ so that $\nabla f(x)\not= 0$, we define: 
 
 $\delta (x):=\delta _j$, where $j$ is the smallest index for which $minsp (\nabla ^2f(x)+\delta _j||\nabla f(x)||^{1+\alpha})\geq ||\nabla f(x)||^{(1+\alpha )}$;
 
  $A(x):=\nabla ^2f(x)+\delta (x)||\nabla f(x)||^{1+\alpha}$; 
  
  $v(x):=A(x)^{-1}.\nabla f(x)$;
  
  $w(x):=pr_{A(x),+}v(x)-pr_{A(x).-}v(x)$; 
  
  and
  
  $\widehat{w(x)}:=w(x)/\max \{1,||w(x)||\}$.

  \begin{lemma} Let $f$ be a $C^2$ function and $x\in \mathbb{R}^m$ for which $\nabla f(x)\not= 0$ We have:

1) $||\nabla f(x)||/sp(A(x)) \leq ||w(x)||\leq ||\nabla f(x)||/minsp(A(x))$.

2) (Descent direction)  

\begin{eqnarray*}
&&||\nabla f(x)||^2/sp(A(x))\leq <w(x),\nabla f(x)>\leq ||\nabla f(x)||^2/minsp(A(x)), \\
&& minsp(A(x))||w(x)||^2\leq <w(x),\nabla f(x)>\leq sp(A(x))||w(x)||^2.
\end{eqnarray*}  

3) In particular,  Armijo's condition 
\begin{eqnarray*}
f(x-\gamma w(x))-f(x)\leq -\gamma <w(x),\nabla f(x)>/2, 
\end{eqnarray*}
is satisfied for all $\gamma >0$ small enough.

\label{Lemma1}\end{lemma}
\begin{proof}
We denote by $e_1,\ldots ,e_m$ an orthonormal basis of eigenvectors of $A(x)$, and let $\lambda _i$ be the corresponding eigenvalue of $e_i$. If we write: 
$$\nabla f(x)=\sum _{i=1}^ma_ie_i=\sum _{\lambda _i>0}a_ie_i+\sum _{\lambda _i<0}a_ie_i,$$
then by definition
\begin{eqnarray*}
v(x)&=&\sum _{\lambda _i>0}a_ie_i/\lambda _i+\sum _{\lambda _i<0}a_ie_i/\lambda _i=\sum _{\lambda _i>0}a_ie_i/|\lambda _i|-\sum _{\lambda _i<0}a_ie_i/|\lambda _i|,\\
w(x)&=&\sum _{\lambda _i>0}a_ie_i/|\lambda _i|+\sum _{\lambda _i<0}a_ie_i/|\lambda _i|=\sum _{i=1}^ma_ie_i/|\lambda _i|. 
\end{eqnarray*}
 
Hence, by calculation 
\begin{eqnarray*}
<w(x),\nabla f(x)>&=&\sum _{i=1}^ma_i^2/|\lambda _i|,\\
||\nabla f(x)||^2&=&\sum _{i=1}^ma_i^2,\\
||w(x)||^2&=&\sum _{i=1}^ma_i^2/|\lambda _i|^2.\\
\end{eqnarray*}

From this, we immediately obtain 1) and  2).

3)  By Taylor's expansion, for $\gamma $ small enough: 
\begin{eqnarray*}
f(x-\gamma w(x))-f(x)&=&-\gamma <w(x),\nabla f(x)> +\frac{\gamma ^2}{2}<\nabla ^2f(x)w(x),w(x)>+o(||\gamma w(x)||^2)\\
&=&-\gamma <w(x),\nabla f(x)>+\gamma ^2O(<w(x),\nabla f(x)>).
\end{eqnarray*}
The last equality follows from 2), where we know that $||w(x)||^2\sim <w(x),\nabla f(x)>$. 
\end{proof}
  
The above lemma shows that, for $x$ with $\nabla f(x)\not= 0$,  the following quantity (Armijo's step-size) is a well-defined positive number: 
  
  $\gamma (x):=$ the largest number $\gamma $ among the sequence $\{1,0.5, 0.5^2,\ldots ,0.5^j,\ldots \}$ for which  
  \begin{eqnarray*}
  f(x-\gamma w(x))-f(x)\leq -\gamma <w(x),\nabla f(x)>/2. 
  \end{eqnarray*}
    
  The next lemma explores properties of this quantity.

  \begin{lemma} Let $\mathcal{B}\subset \mathbb{R}^m$ be a compact set so that $\epsilon =\inf_{x\in \mathcal{B}}||\nabla f(x)||>0$. Then

  1) There is $\tau >0$ for which: 
\begin{eqnarray*}
\sup _{x\in \mathcal{B}}sp(A(x))&\leq& \tau ,\\
\inf _{x\in \mathcal{B}}minsp(A(x))&\geq &1/\tau . 
\end{eqnarray*}

2) $\inf _{x\in \mathcal{B}}\gamma (x)> 0$.  

    \label{Lemma2}\end{lemma}
\begin{proof}

1) Since both $\nabla ^2f(x)$ and $\nabla f(x)$ have bounded norms on the compact set $\mathcal{B}$, and $\delta _j$ belongs to a finite set, it follows easily that $\sup _{x\in \mathcal{B}}sp(A(x))<\infty$. (For example, this is seen by using that for a real symmetric matrix $A$, we have $sp(A)=\max _{||v||=1}|<Av,v>| $.)

By Lemma \ref{Lemma0}, $\inf _{x\in \mathcal{B}}minsp(A(x))\geq \kappa \epsilon ^{1+m}$. 

2) From 1) and Lemma \ref{Lemma1},  there is $\tau >0$ so that for all $x\in \mathcal{B}$ we have
\begin{eqnarray*}
&&||\nabla f(x)||^2/\tau \leq <w(x),\nabla f(x)>\leq \tau ||\nabla f(x)||^2, \\
&& ||w(x)||^2/\tau \leq <w(x),\nabla f(x)>\leq \tau ||w(x)||^2.
\end{eqnarray*}  

From this, using Taylor's expansion as in the proof of part 3) of Lemma \ref{Lemma1} and the assumption that $f$ is $C^2$, it is easy to obtain the assertion that $\inf _{x\in \mathcal{B}}\gamma (x)> 0$. 

\end{proof}  

 Now we prove Theorem \ref{Theorem1}. 
 \begin{proof}[Proof of Theorem \ref{Theorem1}]

0) This is by construction. 

1) It suffices to show that if $\{x_{n_k}\}$ is a bounded subsequence of $\{x_n\}$, then $\lim _{n\rightarrow\infty}\nabla f(x_{n_k})$. If it is not the case, we can assume, by taking a further subsequence if needed, that $\inf _{k}||\nabla f(x_{n_k})||>0$. Then, by Lemma \ref{Lemma2}, upto some positive constants, $<\nabla f(x_{n_k}),w(x_{n_k})>$ $\sim$ $||\nabla f(x_{n_k})||^2$ $\sim$ $||w(x_{n_k})||^2$ for all $k$. Since $||\nabla f(x_{n_k})||$ is bounded, it follows that $w_{n_k}\sim \widehat{w_{n_k}}$.   

It follows from the fact that $\{f(x_n)\}$ is a decreasing sequence, that $\lim _{k\rightarrow\infty}(f(x_{n_k})-f(x_{n_{k+1}}))=0$. It follows from Armijo's condition that $\lim _{k\rightarrow\infty}\gamma _{n_k}<\nabla f(x_{n_k}),\widehat{w(x_{n_k})}>=0$. By Lemma \ref{Lemma2} we have $\inf _{k}\gamma _{n_k}>0$, and hence $\lim _{k\rightarrow\infty}<\nabla f(x_{n_k}),\widehat{w(x_{n_k})}>=0$.  From the equivalence in the previous paragraph about $w_{n_k}\sim \widehat{w_{n_k}}$ and $<\nabla f(x_{n_k}),\widehat{w(x_{n_k})}>\sim ||w_{n_k}||^2\sim ||\nabla f(x_{n_k})||^2$, it follows that $\lim _{k\rightarrow\infty}\nabla f(x_{n_k})=0$, a contradiction.

2)  Assume that $x_n$ converges to $x_{\infty}$ and $x_{\infty}$ is a saddle point. Then $\nabla ^2f(x_{\infty})$ is invertible while $\nabla f(x_{\infty})=0$. It follows that there is a small open neighbourhood $U$ of $x_{\infty}$ so that for all $x\in U$ we have $\delta (x)=\delta _0$. By shrinking $U$ if necessary, we also have that $||w(x)||\leq \epsilon$ for all $x\in U$, and hence $\widehat{w(x)}=w(x)$, for all $x\in U$. Here $\epsilon >0$ is a small positive number, to be determined later. 

By Taylor's expansion, and using $A(x)=\nabla ^2f(x)+\delta _0||\nabla f(x)||^{1+\alpha}Id=\nabla ^2f(x)+o(1)$, we have 
\begin{eqnarray*}
f(x-w(x))-f(x)&=&-<w(x),\nabla f(x)>+\frac{1}{2}<\nabla ^2f(x)w(x),w(x)>+o(||w(x)||^2)\\
&=&-<w(x),\nabla f(x)>+\frac{1}{2}<A(x)w(x),w(x)>+o(||w(x)||^2).
\end{eqnarray*}

We denote by $e_1,\ldots ,e_m$ an orthonormal basis of eigenvectors of $A(x)$, and let $\lambda _i$ be the corresponding eigenvalue of $e_i$. If we write: 
$$\nabla f(x)=\sum _{i=1}^ma_ie_i=\sum _{\lambda _i>0}a_ie_i+\sum _{\lambda _i<0}a_ie_i,$$
then by definition
\begin{eqnarray*}
v(x)&=&\sum _{\lambda _i>0}a_ie_i/\lambda _i+\sum _{\lambda _i<0}a_ie_i/\lambda _i=\sum _{\lambda _i>0}a_ie_i/|\lambda _i|-\sum _{\lambda _i<0}a_ie_i/|\lambda _i|,\\
w(x)&=&\sum _{\lambda _i>0}a_ie_i/|\lambda _i|+\sum _{\lambda _i<0}a_ie_i/|\lambda _i|=\sum _{i=1}^ma_ie_i/|\lambda _i|. 
\end{eqnarray*}

Then $A(x)w(x)=\sum _{\lambda _i>0}a_ie_i-\sum _{\lambda _i<0}a_ie_i$.  Therefore,
\begin{eqnarray*}
\frac{1}{2}<A(x)w(x),w(x)>&=&\frac{1}{2}<\sum _{\lambda _i>0}a_ie_i-\sum _{\lambda _i<0}a_ie_i,\sum _{i=1}^ma_ie_i/|\lambda _i|>\\
&=&\frac{1}{2}(\sum _{\lambda _i>0}a_i^2/|\lambda _i|-\sum _{\lambda _i<0}a_i^2/|\lambda _i|)\\
&\leq&\frac{1}{2}\sum _{i=1}^ma_i^2/|\lambda _i|\\
&=&\frac{1}{2}<w(x),\nabla f(x)>. 
\end{eqnarray*}

Combining all the above, we obtain: 
\begin{eqnarray*}
f(x-w(x))-f(x)\leq -\frac{1}{2}<w(x),\nabla f(x)>+o(||w(x)||^2),
\end{eqnarray*}
for all $x\in U$. By Lemma \ref{Lemma1}, for $x\in U$ we have $<w(x),\nabla f(x)>\sim ||w(x)||^2$. Hence, $\gamma (x)=1$ for all $x\in U$. 

This means that $x_{n+1}=x_n-w_n$ for $n$ large enough, that is, New Q-Newton's method Backtracking becomes New Q-Newton's method. Then the results in \cite{truong-etal, truongnew} finish the proof that $x_{\infty}$ cannot be a saddle point. 

3), 4) and 5): follow easily from the corresponding parts in \cite{truongnew} and 
\end{proof}
 
For the proof of Theorem \ref{Theorem2}, we recall that there is so-called real projective space $\mathbb{P}^m$, which is a compact metric space and which contains $\mathbb{R}^m$ as a {\bf topological} subspace. If $x,y\in \mathbb{R}^m$ and $d(.,.)$ is the metric on   $\mathbb{P}^m$, then 
\begin{eqnarray*}
d(x,y):=\arccos (\frac{1+<x,y>}{\sqrt{1+||x||^2}\sqrt{1+||y||^2}}). 
\end{eqnarray*} 
It is known that there is a constant $C>0$ so that for $x,y\in \mathbb{R}^m$ we have $d(x,y)\leq C||x-y||$, see e.g. \cite{truong-nguyen2}. Real projective spaces were used in \cite{truong-nguyen1, truong-nguyen2} to establish good convergence guarantees for Backtracking GD and modifications. 

\begin{proof}[Proof of Theorem \ref{Theorem2}]

We show first that $\lim _{n\rightarrow\infty}d(x_n,x_{n+1})=0$, where $d(.,.)$ is the projective metric defined above.  It suffices to show that each subsequence $\{x_{n_k}\}$ has another subsequence for which the needed claim holds. By taking a further subsequence if necessary, we can assume that either $\lim _{k\rightarrow\infty}||x_{n_k}||=\infty$ or $\lim _{k\rightarrow\infty}x_{n_k}=x_{\infty}$ exists. 

The first case:  $\lim _{k\rightarrow\infty}||x_{n_k}||=\infty$. In this case, since $x_{n_{k}+1}=x_{n_k}-\gamma _{n_k}\widehat{w_{n_k}}$, where both $\gamma _{n_k}$ and $\widehat{w_{n_k}}$ are bounded, it follows easily from the definition of the projective metric that $\lim _{k\rightarrow \infty}d(x_{n_k+1},x_{n_k})=0$. 

The second case: $\lim _{k\rightarrow\infty}x_{n_k}=x_{\infty}$ exists. In this case, by part 1) of Theorem \ref{Theorem1}, $x_{\infty}$ is a critical point of $f$.  Since $f$ is Morse by assumption, we have that $\nabla ^2f(x_{\infty})$ is invertible. Then, by the proof of part 2) of Theorem \ref{Theorem1}, we have that $\lim _{k\rightarrow\infty}\widehat{w_{n_k}}=0$. Then, since $\gamma _{n_k}$ is bounded, one obtain that
\begin{eqnarray*}
\lim _{k\rightarrow\infty}||x_{n_k+1}-x_{n_k}||=\lim _{k\rightarrow\infty}||\gamma _{n_k}\widehat{w_{n_k}}||=0. 
\end{eqnarray*}
Thus by the inequality between the projective metric and the usual Euclidean norm, one obtains $\lim _{k\rightarrow\infty}d(x_{n_k+1},x_{n_k})=0$. 

With this claim proven, one can proceed as in \cite{truong-nguyen2}, using part 1) of Theorem \ref{Theorem1} and the fact that the set of critical points of a Morse function is at most countable, to obtain a bifurcation: 

either

i) $\lim _{n\rightarrow\infty}||x_n||=\infty$

or 

ii) $\lim _{n\rightarrow\infty}x_n=x_{\infty}$ exists, and $x_{\infty}$ is a critical point of $f$. 

If $f$ has compact sublevels, then since $f(x_n)\leq f(x_0)$ for all $n$, only case ii) can happen. 

In case ii): By parts 2) and 3)  of Theorem \ref{Theorem1}, $x_{\infty}$ is a local minimum and the rate of convergence is quadratic. 
 \end{proof} 
\subsection{Variants} We present some variants in this subsection. 

First we observe that the $\widehat{w_k}$ in the definition of New Q-Newton's method is bounded. This could make slow convergence in certain cases, see experiments in the next section. We note that the attraction of Newton's method lies in that when it converges then it converges fast, and this is achieved by using $(\nabla ^{2}f)^{-1}\nabla f$ (when definable), which can be unbounded even in compact sets.  This suggests that we can use directly the direction $w_k$ in the line search in the definition of New Q-Newton's method Backtracking. This way, we obtain a variant named New Q-Newton's method Backtracking S, see Algorithm \ref{table:alg0}. 

\medskip
{\color{blue}
 \begin{algorithm}[H]
\SetAlgoLined
\KwResult{Find a critical point of $f:\mathbb{R}^m\rightarrow \mathbb{R}$}
Given: $\{\delta_0,\delta_1,\ldots, \delta_{m}\} \subset \mathbb{R}$\ (chosen {\bf randomly}, and $\alpha >0$;
Initialization: $x_0\in \mathbb{R}^m$\;
$\kappa:=\frac{1}{2}\min _{i\not=j}|\delta _i-\delta _j|$;\\
 \For{$k=0,1,2\ldots$}{ 
    $j=0$\\
  \If{$\|\nabla f(x_k)\|\neq 0$}{
   \While{$minsp(\nabla^2f(x_k)+\delta_j \|\nabla f(x_k)\|^{1+\alpha}Id)<\kappa  \|\nabla f(x_k)\|^{1+\alpha}$}{$j=j+1$}}

 $A_k:=\nabla^2f(x_k)+\delta_j \|\nabla f(x_k)\|^{1+\alpha}Id$\\
$v_k:=A_k^{-1}\nabla f(x_k)=pr_{A_k,+}(v_k)+pr_{A_k,-}(v_k)$\\
$w_k:=pr_{A_k,+}(v_k)-pr_{A_k,-}(v_k)$\\
$\gamma :=1$\\
 \If{$\|\nabla f(x_k)\|\neq 0$}{
   \While{$f(x_k-\gamma {w_k})-f(x_k)>-\gamma <{w_k},\nabla f(x_k)>/2$}{$\gamma =\gamma /2$}}

$x_{k+1}:=x_k-\gamma {w_k}$
   }
  \caption{New Q-Newton's method Backtracking S} \label{table:alg0}
\end{algorithm}
}
\medskip

New Q-Newton's method Backtracking S almost satisfies the same theoretical guarantees as New Q-Newton's method Backtracking. Indeed, we have the following.

\begin{theorem}

Theorem \ref{Theorem1} holds if we replace  New Q-Newton's method Backtracking by New Q-Newton's method Backtracking S. 

Theorem \ref{Theorem2} holds if we replace  New Q-Newton's method Backtracking by New Q-Newton's method Backtracking S, and replace the phrase  "either $ \lim _{n\rightarrow\infty}||x_n||=\infty$" by   "either $ \limsup _{n\rightarrow\infty}||x_n||=\infty$". 

\label{TheoremS}\end{theorem}  
\begin{proof}
Concerning Theorem \ref{Theorem1}: The most crucial part is to show that if $\{x_{n_k}\}$ converges to $x_{\infty}$, then $x_{\infty}$ is a critical point of $f$. To this end, we note that the proof of part 0) of Theorem \ref{Theorem1} uses that when $\mathcal{B}$ is a compact set so that all $\inf _{x\in \mathcal{B}}||\nabla f(x)||>0$, then $w(x)$ and $\widehat{w(x)}$ are equivalence for all $x\in \mathcal{B}$. Hence, when we replace $\widehat{w_k}$ in the proof by $w_k$, we obtain the same conclusions.   
 
 Concerning Theorem \ref{Theorem2}: The most crucial part is to show that $\lim _{n\rightarrow\infty}d(x_{n},x_{n+1})=0$. We need $\widehat{w_k}$ bounded to deal with the case when $x_{n_k}$ diverges to $\infty$. For the case  $\{x_{n_k}\}$ converges to a point $x_{\infty}$, we use that $x_{\infty}$ is a non-degenerate critical point to obtain again the  equivalence between $w_{n_k}$ and $\widehat{w_{n_k}}$, and hence obtain the same conclusions when in the proofs replacing $\widehat{w_k}$ by $w_k$. Thus, if we replace "either $x_n$ diverges to $\infty$" by   "$ either \limsup _{n\rightarrow\infty}||x_n||=\infty$", then we can replace $\widehat{w_k}$ by $w_k$ and obtain the same conclusions. 

\end{proof}

While New Q-Newton's method Backtracking and New Q-Newton's method Backtracking S are descriptive enough to be straight forwardly implemented in (Python) codes, the first While loop in their definitions may be time consuming to check. The purpose of that While loop is to assure that part 1) of Lemma \ref{Lemma2} is satisfied. Indeed, we only need part 1) of Lemma \ref{Lemma2} for bounded subsequences of the sequence $\{x_n\}$. Even if we do not have $minsp (A+\delta \epsilon Id)\geq \kappa \epsilon $ (where $\epsilon =||\nabla f(x)||^{1+\alpha}$ and $A=\nabla ^2f(x)$) uniformly across the sequence, it may still happen that part 1) of Lemma \ref{Lemma2} holds for bounded subsequences of the sequence $\{x_n\}$ if we choose $\delta _j$'s randomly. Hence, by replacing it with the less demanding first While loop in the definition of New Q-Newton's method, we obtain the variant New Q-Newton's method Backtracking V2 below. One can argue further, that since the step direction $\widehat{w_k}$ already has some good properties, it might be not necessary to do the whole Armijo's search, but a less demanding line search $f(x_k-\delta \widehat{w_k})-f(x_k)\leq 0$. This way, one obtains the variant New Q-Newton's method Backtracking V1 below. If, like what done in the case of New Q-Newton's method Backtracking S, in the line search component we replace $\widehat{w_k}$ by $w_k$, then we obtain the variants New Q-Newton's method Backtracking V3 and New Q-Newton's method Backtracking V4.

\medskip
{\color{blue}
 \begin{algorithm}[H]
\SetAlgoLined
\KwResult{Find a critical point of $f:\mathbb{R}^m\rightarrow \mathbb{R}$}
Given: $\{\delta_0,\delta_1,\ldots, \delta_{m}\}\subset \mathbb{R}$\ (chosen {\bf randomly}) and $\alpha >0$;
Initialization: $x_0\in \mathbb{R}^m$\;
 \For{$k=0,1,2\ldots$}{ 
    $j=0$\\
    \If{$\|\nabla f(x_k)\|\neq 0$}{
   \While{$\det(\nabla^2f(x_k)+\delta_j \|\nabla f(x_k)\|^{1+\alpha}Id)=0$}{$j=j+1$}}

$A_k:=\nabla^2f(x_k)+\delta_j \|\nabla f(x_k)\|^{1+\alpha}Id$\\
$v_k:=A_k^{-1}\nabla f(x_k)=pr_{A_k,+}(v_k)+pr_{A_k,-}(v_k)$\\
$w_k:=pr_{A_k,+}(v_k)-pr_{A_k,-}(v_k)$\\
$\widehat{w_k}:={w_k}/\max\{1,||w_k||\}$\\
$\gamma :=1$\\
 \If{$\|\nabla f(x_k)\|\neq 0$}{
   \While{$f(x_k-\gamma \widehat{w_k})-f(x_k)>0$}{$\gamma =\gamma /2$}}

$x_{k+1}:=x_k-\gamma \widehat{w_k}$
   }
  \caption{New Q-Newton's method Backtracking V1} \label{table:alg1}
\end{algorithm}
}

\medskip
{\color{blue}
 \begin{algorithm}[H]
\SetAlgoLined
\KwResult{Find a critical point of $f:\mathbb{R}^m\rightarrow \mathbb{R}$}
Given: $\{\delta_0,\delta_1,\ldots, \delta_{m}\} \subset \mathbb{R}$\ (chosen {\bf randomly}) and $\alpha >0$;
Initialization: $x_0\in \mathbb{R}^m$\;
 \For{$k=0,1,2\ldots$}{ 
    $j=0$\\
    \If{$\|\nabla f(x_k)\|\neq 0$}{
   \While{$\det(\nabla^2f(x_k)+\delta_j \|\nabla f(x_k)\|^{1+\alpha}Id)=0$}{$j=j+1$}}

$A_k:=\nabla^2f(x_k)+\delta_j \|\nabla f(x_k)\|^{1+\alpha}Id$\\
$v_k:=A_k^{-1}\nabla f(x_k)=pr_{A_k,+}(v_k)+pr_{A_k,-}(v_k)$\\
$w_k:=pr_{A_k,+}(v_k)-pr_{A_k,-}(v_k)$\\
$\widehat{w_k}:=w_k/\max \{1,||w_k||\}$\\
$\gamma :=1$\\
 \If{$\|\nabla f(x_k)\|\neq 0$}{
   \While{$f(x_k-\gamma \widehat{w_k})-f(x_k)>-\gamma <\widehat{w_k},\nabla f(x_k)>/2$}{$\gamma =\gamma /2$}}

$x_{k+1}:=x_k-\gamma \widehat{w_k}$
   }
  \caption{New Q-Newton's method Backtracking V2} \label{table:alg2}
\end{algorithm}
}
\medskip

\medskip
{\color{blue}
 \begin{algorithm}[H]
\SetAlgoLined
\KwResult{Find a critical point of $f:\mathbb{R}^m\rightarrow \mathbb{R}$}
Given: $\{\delta_0,\delta_1,\ldots, \delta_{m}\}\subset \mathbb{R}$\ (chosen {\bf randomly}) and $\alpha >0$;
Initialization: $x_0\in \mathbb{R}^m$\;
 \For{$k=0,1,2\ldots$}{ 
    $j=0$\\
    \If{$\|\nabla f(x_k)\|\neq 0$}{
   \While{$\det(\nabla^2f(x_k)+\delta_j \|\nabla f(x_k)\|^{1+\alpha}Id)=0$}{$j=j+1$}}

$A_k:=\nabla^2f(x_k)+\delta_j \|\nabla f(x_k)\|^{1+\alpha}Id$\\
$v_k:=A_k^{-1}\nabla f(x_k)=pr_{A_k,+}(v_k)+pr_{A_k,-}(v_k)$\\
$w_k:=pr_{A_k,+}(v_k)-pr_{A_k,-}(v_k)$\\
$\gamma :=1$\\
 \If{$\|\nabla f(x_k)\|\neq 0$}{
   \While{$f(x_k-\gamma {w_k})-f(x_k)>0$}{$\gamma =\gamma /2$}}

$x_{k+1}:=x_k-\gamma {w_k}$
   }
  \caption{New Q-Newton's method Backtracking V3} \label{table:alg3}
\end{algorithm}
}

\medskip
{\color{blue}
 \begin{algorithm}[H]
\SetAlgoLined
\KwResult{Find a critical point of $f:\mathbb{R}^m\rightarrow \mathbb{R}$}
Given: $\{\delta_0,\delta_1,\ldots, \delta_{m}\} \subset \mathbb{R}$\ (chosen {\bf randomly}) and $\alpha >0$;
Initialization: $x_0\in \mathbb{R}^m$\;
 \For{$k=0,1,2\ldots$}{ 
    $j=0$\\
    \If{$\|\nabla f(x_k)\|\neq 0$}{
   \While{$\det(\nabla^2f(x_k)+\delta_j \|\nabla f(x_k)\|^{1+\alpha}Id)=0$}{$j=j+1$}}

$A_k:=\nabla^2f(x_k)+\delta_j \|\nabla f(x_k)\|^{1+\alpha}Id$\\
$v_k:=A_k^{-1}\nabla f(x_k)=pr_{A_k,+}(v_k)+pr_{A_k,-}(v_k)$\\
$w_k:=pr_{A_k,+}(v_k)-pr_{A_k,-}(v_k)$\\
$\gamma :=1$\\
 \If{$\|\nabla f(x_k)\|\neq 0$}{
   \While{$f(x_k-\gamma {w_k})-f(x_k)>-\gamma <{w_k},\nabla f(x_k)>/2$}{$\gamma =\gamma /2$}}

$x_{k+1}:=x_k-\gamma {w_k}$
   }
  \caption{New Q-Newton's method Backtracking V4} \label{table:alg4}
\end{algorithm}
}
\medskip

\section{Experiments}
\subsection{Implementation details} Interested readers are referred to \cite{truong-etal} for a description on implementation details for New Q-Newton's method in Python, and to the GitHub link \cite{phuongGitHub} for Python codes for New Q-Newton's method. With that, it is easy to implement New Q-Newton's method Backtracking V1--4 by adding a few lines for checking Armijo's condition $f(x_n-\gamma \widehat{w_n})-f(x_n)$ $\leq$ $-\gamma <\widehat{w_n},\nabla f(x_n)>/2$ or the weaker descent property $f(x_n-\gamma \widehat{w_n})\leq f(x_n)$. Codes for New Q-Newton's method Backtracking V1--4 are available at the GitHub link \cite{tuyenGitHub}. 

Using the ideas from \cite{truong-nguyen1, truong-nguyen2}, one can also incorporate the Two-way Backtracking version in the line search component for faster performance, but to keep things clear in this paper we consider only the above simple line search versions.  Similar to New Q-Newton's method, one does not need precise values of the gradient and Hessian, but approximations are good enough. 

\subsection{Experimental results} In \cite{truong-etal}, we have tested New Q-Newton's method with various problems, including a toy model for protein, a stochastic version of Griewank's function, and various benchmark cost functions. While generally New Q-Newton's method has a very good performance in comparison to other methods, there are cases (in particular, for some non-smooth functions or some not good initial points), its performance is not so good. We have redone many of these experiments with New Q-Newton's method Backtracking (in particular the cases where New Q-Newton's method does not perform well), and found that indeed the performance is similar or improved. In the below, to keep the space reasonable, we will report only several cases where the improvement by New Q-Newton's method Backtracking is most significantly (either in terms of computing resources, running time or convergence to a better point).  

The experimental setting is as in \cite{truong-etal}. We use the python package numdifftools \cite{num} to compute approximately the gradients and Hessian of a real function, since symbolic computation is not quite efficient. All experiments are run on a small personal laptop. (Note: In \cite{truong-etal}, for the cases $N_n=500$ and $N_n=1000$ in Table \ref{tab:Griewank2} it needs  a stronger computer  with configuration Processor: Intel Core i9-9900X CPU @ 3.50GHzx20, Memory: 125.5 GiB. Here we can run on a usual personal laptop, and hence needs less computing resources.) The unit for running time is seconds.

Here, we will compare the performance of New Q-Newton's method Backtracking against New Q-Newton's method, the usual Newton's method, BFGS, Adaptive Cubic Regularization \cite{nesterov-polyak, cartis-etal}, as well as Random damping Newton's method \cite{sumi}  and Inertial Newton's method \cite{bolte-etal}.
 
For New Q-Newton's we choose $\alpha =1$ in Algorithm \ref{table:alg}. Moreover, we will choose $\Delta =\{0,\pm 1\}$, even though for theoretical proofs we need $\Delta$ to have at least $m+1$ elements, where $m=$ the number of variables. The justification is that when running New Q-Newton's method it almost never happens the case that both $\nabla ^2f(x)$ and $\nabla ^2f(x)\pm ||\nabla f(x)||^2Id$ are not invertible. These parameters are used for New Q-Newton's method Backtracking V1--4 as well. For BFGS: we use the function scipy.optimize.fmin$\_$bfgs available in Python, and put  $gtol=1e-10$ and $maxiter=1e+6$. For Adaptive cubic regularization for Newton's method, we use the AdaptiveCubicReg module in the implementation in \cite{ARCGitHub}. We use the default hyperparameters as recommended there, and use "exact" for the hessian$\_$update$\_$method.  For hyperparameters in Inertial Newton's method, we choose $\alpha =0.5$ and $\beta =0.1$ as recommended by the authors of \cite{bolte-etal}. 

We will also compare the performance to Unbounded Two-way Backtracking GD \cite{truong-nguyen1}. The hyperparameters for Backtracking GD are fixed through all experiments as follows: $\delta _0=1$, $\alpha =0.5$ and $\beta =0.7$. Recall that this means we have the following in Armijo's condition: $f(x-\beta ^m\delta _0x)-f(x)\leq -\alpha \beta ^m\delta _0||\nabla f(x)||^2$, where $m\in \mathbb{Z}_{\geq 0}$ depends on $x$. Here we recall the essence of Unbounded and Two-way variants of Backtracking GD, see \cite{truong-nguyen1} for more detail. In the Two-way version, one starts the search for learning rate $\delta _n$ - at the step n- not at $\delta _0$ but at $\delta _{n-1}$, and allows the possibility of increasing $\delta \mapsto \delta /\beta $, and not just decreasing $\delta \mapsto \delta \beta$ as in the standard version of Backtracking GD. In the Unbounded variant, one allows the upper bound for $\delta _n$ not as $\delta _0$ but as $\max\{\delta _0,\delta _0||\nabla f(x_n)||^{-\kappa}\}$ for some constant $0<\kappa <1$. In all the experiments here, we fix $\kappa =1/2$. The Two-way version helps to reduce the need to do function evaluations in checking Armijo's condition, while the Unbounded version helps to make large step sizes near degenerate critical points (e.g. the original point for the function $f(x,y)=x^4+y^4$) and hence also helps with quicker convergence. 

{\bf Legends:} We use the following abbreviations:  "ACR" for Adaptive cubic regularisation, "BFGS" for itself, "Newton" for Newton's method, "Iner" for Inertial Newton's method, "NewQ" for New Q-Newton's method, "V1" for New Q-Newton's method Backtracking V1, and similarly for "V2", "V3", "V4", and "Back" for Unbounded Two-way Backtracking GD.  

{\bf Features reported:} We will report on the number of iterations needed, the function value and the norm of the gradient at the last point, as well as the time needed to run. 

{\bf Remarks:} In most of experiments we will only use "V1" and "V2", but in the Stochastic Griewank's test function, which is very time consuming and resource demanding, we use also "V3" and "V4". 

\subsubsection{A toy model for protein folding} This problem is taken from \cite{shh}. Here is a brief description of the problem. The model has only two amino acids, called A and B, among 20 that occurs naturally. A molecule with n amino acids will be called an n-mer. The amino acids will be linked together and determined by the angles of bend $\theta _2,\ldots ,\theta _{n-1}\in [0,2\pi ]$. We specify the amino acids by boolean variables $\xi _1,\ldots ,\xi _n\in \{1,-1\}$, depending on whether the corresponding one is A or B. The intramolecular potential energy is given by: 
\begin{eqnarray*}
\Phi =\sum _{i=2}^{n-1}V_1(\theta _i)+\sum _{i=1}^{n-2} \sum_{j=i+2}^n V_2(r_{i,j},\xi _i,\xi _j).
\end{eqnarray*}
Here $V_1$ is the backbone bend potential and $V_2$ is the nonbonded interaction, given by:  

\begin{eqnarray*}
V_1(\theta _i)&=&\frac{1}{4}(1-\cos (\theta _i)),\\
r_{i,j}^2&=&[\sum _{k=i+1}^{j-1}\cos (\sum _{l=i+1}^k\theta _l)]^2+[\sum _{k=i+1}^{j-1}\sin (\sum _{l=i+1}^k\theta _l)]^2,\\
C(\xi _i,\xi _j)&=&\frac{1}{8}(1+\xi _i+\xi _j+5\xi _i\xi _j),\\
V_2(r_{i,j},\xi _i,\xi _j)&=&4(r_{i,j}^{-12}-C(\xi _i,\xi _j)r_{i,j}^{-6}). 
\end{eqnarray*}
Note that the value of $C(\xi _i,\xi _j)$ belongs to the finite set $\{1,0.5,-0.5\}$. 

It is explained in \cite{truong-etal} that Table 1 in \cite{shh} seems contain inaccurate results. We will report, see Table  the experimental result for the case of ABBBABABAB, with the initial point (randomly chosen)
\begin{eqnarray*}
&&(\theta _2,\theta _3,\theta _4, \theta _5,\theta _6,\theta _7, \theta _8)\\
&=&(-1.3335047,   2.76782837, -1.89518385,  2.52345111,\\
&&-0.33519698, -1.98794015,
  0.02088706, -1.09200044).  
\end{eqnarray*}
The function value at the initial point is $579425.218039767$. In this case, New Q-Newton's method performs poorly.  See Table \ref{tab:ProteinFolding10}. 

\begin{table}[htp]
\fontsize{11}{11}\selectfont
  \centering
  \begin{tabular}{|l|c|c|c|c|c|c|c|c|}
  \hline
&ACR&BFGS	& Newton   & NewQ  & V1&V2& Iner& Back\\
%\hline
%&\multicolumn{7}{c|}{}\\
\hline
Iterations &80&76& 50 & 330 &36&36& 13&65 \\
\hline
$f$ &4e+19&19.587&9.7e+4 & 1e+3 &19.427 &19.427 &  6e+11&20.225 \\
\hline
$||\nabla f||$ &7e+11& 1e-7 & 2.4e+6 &5e-9 &5e-11& 5e-11& 0&2e-7\\
\hline
Time &4.781&3.850&14.823 & 100 & 14.665&14.393 & 0.481&21.996 \\
\hline

  \end{tabular}
   \caption{Performance of different optimization methods for the toy protein folding problem for the 10-mer ABBBABABAB at a randomly chosen initial point. The function value at the initial point is 579425.218039767. For Newton's method: we have to take a special care and reduce the number of iterations to 50.  }%  
  \label{tab:ProteinFolding10}
\end{table}

\subsubsection{Stochastic optimization for Griewank's function} Griewank's function is a well known test function in global optimization. It has the form: 
\begin{eqnarray*}
f(x_1,\ldots ,x_m)=1+\frac{1}{4000}\sum _{i=1}^mx_i^2-\prod _{i=1}^m\cos (x_i/\sqrt{i}). 
\end{eqnarray*}
It has a unique global minimum at the point $(0,\ldots ,0)$, where the function value is 0. The special property of it is that, in terms of the dimension $m$,  it has exponentially many local minima. However, \cite{locatelli} explained that indeed when the dimension increases, it can become more and more easier to find the global minimum.

Here we consider the stochastic version of Griewank's function, following \cite{kk}. Unlike the deterministic case, it becomes increasingly difficult to optimise when the dimension increases. We briefly recall some generalities of stochastic optimization. One  considers a function $f(x,\xi )$, which besides a variable $x$, also depends on a random parameter $\xi $. One wants to optimize the expectation of $f(x,\xi )$: Find $\min _{x}F(x)$, where $F(x)=E(f(x,\xi ))$. 

Assume now that one has an optimization method $A$ to be used to the above problem. Since computing the expectation is unrealistic in general, what one can do is as follows: 

At step n: choose randomly $N_n$ parameters $\xi _{n,1},\ldots ,\xi _{n,N}$, where $N_n$ can depend on $n$ and is usually chosen to be large enough. Then one approximates $F(x)$ by 
\begin{eqnarray*}
F_n(x)=\frac{1}{N_n}\sum _{i=1}^{N_n}f(x,\xi _{n,i}). 
\end{eqnarray*}

This is close to the mini-batch practice in Deep Learning, with a main difference is that in Deep Learning one trains a DNN on a large but finite set of data, and at each step n (i.e. epoch n) decompose the training set randomly into mini-batches to be used. The common practice is to use mini-batches of the same size $N_n=N$ fixed from beginning. There are also experiments with varying/increasing the mini-batch sizes, however this can be time consuming while not archiving competing results. There are also necessary modifications (such as rescaling of learning rates) in the mini-batch practice to obtain good performance, however in the experiments below we do not impose this to keep things simple.   

In this subsection we perform experiments on the stochastic version of the Griewank test function considered in the previous subsection. This problem was considered in \cite{kk}, where the dimension of $x=(x_1,\ldots ,x_m)$ is $m=10$ and of $\xi$ is $1$ (with the normal distribution $N(1,\sigma ^2)$), and $f(x,\xi)$ has the form: 
\begin{eqnarray*}
f(x,\xi )=1+\frac{1}{4000}||\xi x||^2-\prod _{i=1}^m\cos (x_i \xi /\sqrt{i}). 
\end{eqnarray*}
At each step, \cite{kk} chooses $N_n$ varying in an interval $[N_{min},N_{max}]$ according to a complicated rule. Here, to keep things simple, we do not vary $N_n$ but fix it as a constant from beginning. Also, we do not perform experiments on BFGS and Adaptive Cubic Regularization in the stochastic setting, because the codes of these algorithms are either not available to us or too complicated and depend on too many hyperparameters (and the performance is very sensitive to these hyperparameters) to be succesfully changed for the stochastic setting. We note however that the BFGS was tested in \cite{kk}, and interested readers can consult Table 8 in that paper for more detail. 

The settings in \cite{kk} are as follows: the dimension is 10, the $\sigma$ is chosen between 2 values $\sqrt{0.1}$ (with $N_{max}=500$) and $\sqrt{1}$ (with $N_{max}=1000$). We will also use these parameters in the below, for ease of comparison. We note that in \cite{kk}, time performance is not reported but instead the number of function evaluations (about 1.8 million for $\sigma=\sqrt{0.1}$, and about 6.3 million for $\sigma=\sqrt{1}$). Also, only the norm of gradient was reported in \cite{kk}, in which case it is ranged from $0.005$ to $0.01$. Experimental results are presented in Table \ref{tab:Griewank2}. 

\begin{table}[htp]
\fontsize{10}{10}\selectfont
  \centering
  \begin{tabular}{|l|c|c|c|c|c|c|c|c|}
  \hline
 	& Newton   & NewQ  & V1& V2&V3&V4& Iner& Back\\
\hline
&\multicolumn{8}{c|}{$N_n=10$, $\sigma =\sqrt{0.1}$}\\
\hline

Iterations &1000  & 33 &1e+3&136&52&314& 14& 1000\\
\hline
$f$ &3.8e+18& 0 &0.900 &2e-13& 0&0&2.2e+31 &1.079 \\
\hline
$||\nabla f||$  &6.2e+7  & 0&0.039&2e-7 &0&4e-14&0 &0.008\\
\hline
Time & 587& 17.70&522 &71 &34.23&190& 0.667&612 \\

\hline
&\multicolumn{8}{c|}{$N_n=10$, $\sigma =\sqrt{1}$}\\
\hline

Iterations &1000  & 1000 &85&63&43&146& 13& 486\\
\hline
$f$ &5.3e+19& 1.7e+21 & 0& 0&0&0& 6.6e+29 &5.3e-14 \\
\hline
$||\nabla f||$  &2.9e+8  & 2.1e+9&0&7e-10&0&1e-12 &0 &1.1e-7\\
\hline
Time & 546& 503&44 & 33&25.84&85.81& 0.600&258 \\

%\hline
%&\multicolumn{5}{c|}{}\\

\hline
&\multicolumn{8}{c|}{$N_n=100$, $\sigma =\sqrt{0.1}$}\\
\hline

Iterations &1000  & 1000 &1e+3&1e+3&12 &22&13& 1000\\
\hline
$f$ &2.5e+17& 2.3e+18 &0.941 &0.934 &0&0& 4.6e+29 &0.967 \\
\hline
$||\nabla f||$  &1.7e+7  & 5.2e+7&0.027&0.001 &6e-16&8e-15&0 &0.002\\
\hline
Time & 3.7e+3& 1e+4&1.9e+4 &4.1e+4 &56.71 &105&4.333&5.5e+3 \\
\hline

&\multicolumn{8}{c|}{$N_n=100$, $\sigma =\sqrt{1}$}\\
\hline

Iterations &1000 & 1000 &59&52&52&26& 13& 395\\
\hline
$f$ &6.3e+18& 7.8e+16 &0 & 3e-15&0 &1e-13&9.3e+29 &9.5e-13 \\
\hline
$||\nabla f||$  &1.1e+8  & 1.1e+8&8e-14&5e-8 &0&4e-7&0 &6.3e-7\\
\hline
Time & 3.6e+3& 6.8e+3&243 &215 &246&121& 5.019&3.7e+3 \\
\hline

&\multicolumn{8}{c|}{$N_n=500$, $\sigma =\sqrt{0.1}$}\\
\hline

Iterations &1000 & 14 &52&1e+3& 9&24&13& 1000\\
\hline
$f$ &8.6e+17& 0 & 0&0.976 &0 &4e-13&4.4e+29 &0.964 \\
\hline
$||\nabla f||$  &3.1e+7  & 3.7e-16&1e-16&0.002 &0&3e-7&0 &0.014\\
\hline
Time & 1.0e+4& 142.838&1.1e+3 &2.6e+4&202 &5.5e+3&12.221&1.1e+4 \\
\hline

&\multicolumn{8}{c|}{$N_n=500$, $\sigma =\sqrt{1}$}\\
\hline

Iterations &1000 & 1000 &57&44&10 &24&14& 361\\
\hline
$f$ &2.6e+18& 9.8e+18 & 0&2e-14 & 0&2e-13&3e+21 &6.6e-13 \\
\hline
$||\nabla f||$  &7.3e+7  & 1.3e+8&9e-17&1e-7 &6e-16&3e-7&0 &9.9e-7\\
\hline
Time & 9.9e+3& 9.6e+3&3e+3 &5e+3 &221 &724&12.324&3.7e+3 \\
\hline

&\multicolumn{8}{c|}{$N_n=1000$, $\sigma =\sqrt{0.1}$}\\
\hline

Iterations &1000 & 19 &50&//&9&//& 13& 1000\\
\hline
$f$ &2.1e+17& 0 & 0& //&0&//& 4.6e+29 &0.945 \\
\hline
$||\nabla f||$  &1.5e+7  & 1.6e-16&6e-17&//&1e-16&// &0 &0.003\\
\hline
Time & 20e+3& 365& 2.0e+3& //&373&//& 23.303&21e+3 \\
\hline

&\multicolumn{8}{c|}{$N_n=1000$, $\sigma =\sqrt{1}$}\\
\hline

Iterations &1000 & 1000 &43&46& 15&22&14& 347\\
\hline
$f$ &1.9e+20& 1.7e+18 &0&0 &0&1e-13& 3e+31 &2.5e-12 \\
\hline
$||\nabla f||$  &6.2e+8  & 5.9e+7&2e-16&3e-10 &0&3e-7&0 &1.0e-6\\
\hline
Time & 20e+3& 19e+3&1.7e+3&2e+3&644 &884&25&7.3e+3 \\
\hline

  \end{tabular}
   \caption{ Performance of different optimization methods for {\bf stochastic} Griewank function. Dimension =  $10$, the initial point  $(10,\ldots ,10)$. The function value of the deterministic Griewank test function at the initial point is $1.264$. Mini-batch size $N_n$ is fixed in every steps of each experiment. The results for V1--4 are run on a usual personal laptop. The results for other methods are taken from \cite{truong-etal}, where the cases $N_n=500$ and $1000$ have to be run on a stronger computer.  "//": the experiment does not yield results after long time.} %  
  \label{tab:Griewank2}
\end{table}

\subsubsection{Some benchmark functions}\label{Section:BenchmarkFunctions} We consider several benchmark functions having many different behaviours. See Tables \ref{tab:BenchmarkFunctions} and \ref{tab:BenchmarkFunctions2}. 

We consider a function: $f_1(x)=|x|^{1+1/3}$ (non-smooth at 0). It is not smooth at $x=0$, which is its unique critical point and also its global minimum. The initial point is (randomly chosen) $x_0=1$, at which the function value is $1$.  

We consider the function $f_2(x)=x^3\sin (1/x)$ (non-smooth at 0), which has compact sublevels. It has countably many critical points, which converge to the singular point $0$. The initial point is (randomly chosen) $x_0=0.75134554$, at which the function value is $0.41200773$.  

We consider the function $f_3(x,y)=100(y-|x|)^2+|1-x|$ (not smooth at lines $x=0$ and $x=1$), taken from \cite{bolte-etal}. It has a unique global minimum at $(1,1)$. The initial point (randomly chosen) is $(-0.99998925, 2.00001188)$, at which the function value is $102.004$. The speciality of this function is that New Q-Newton's method seems to enter a "near" cycle $(0.4987,0.5012)$ $\mapsto $ $(1.00124,0.9987)$ $\mapsto$ $(0.4987,0.5012)$. One point in this cycle is close to the global minimum, but the other point is very far away. If we choose another initial point, we see the same behaviour for New Q-Newton's method. 

We consider the Ackley function
\begin{eqnarray*}
f_4(x_1,\ldots ,x_D)=-20*exp[-0.2*\sqrt{0.5\sum _{i=1}^Dx_i^2}]-exp[0.5*\sum _{i=1}^D\cos (2\pi x_i)]+e+20. 
\end{eqnarray*}
The global minimum is at the origin $(0,\ldots, 0)$. We choose $D=3$ and the (randomly chosen) initial point $(0.01,0.02,-0.07)$, at which the function value is $0.262$. 

We consider the Rastrigin function 
\begin{eqnarray*}
f_5(x_1,\ldots ,x_D)=A*D+\sum _{i=1}^D(x_i^2-A\cos (2\pi x_i)).
\end{eqnarray*}
The global minimum is at the origin $(0,\ldots ,0)$. We choose $A=10$, $D=4$. The (randomly chosen) initial point is $(-4.66266579,-2.69585675,-3.08589085,-2.25482451)$, at which the function value is 83.892.

We consider Beale's function 
\begin{eqnarray*}
f_6(x,y)=(1.5-x+xy)^2+(2.25-x-xy^2)^2+(2.625-x-xy^3)^2.
\end{eqnarray*}
The global minimum is at $(3,0.5)$. The (randomly chosen) initial point is $(-0.52012358, -1.28227229)$, at which the function value is 28.879. 

We consider Bukin function $\#6$: 
\begin{eqnarray*}
f_7(x,y)=100\sqrt{|y-0.01x^2}+|0.01x+10|,
\end{eqnarray*}
which is non-smooth on two curves $x=-10$ and $y=0.01x^2$. The global minimum is at $(-10,1)$, with function value 0. We will consider 2 (randomly chosen) initial points. Point 1: $(4:38848192,-3.47943683)$, at which the function value is $191.769$. Point 2: $(-9.7,0.7)$, closer to the global minimum, at which the function value is 49.084. 

%We consider L\'evi function $\#13$: 
%\begin{eqnarray*}
%f_8(x,y)=\sin^2(3\pi x)+(x-1)^2\times (1+\sin^2(3\pi y))+(y-1)^2(1+\sin ^2(2\pi y)).
%\end{eqnarray*}
%The  global minimum is at $(1,1)$. The (randomly chosen) initial point is $(-3.52914182, 1.36683019)$, at which the function value is 23.625. 

We consider Schaffer function $\#2$:
\begin{eqnarray*}
f_8(x,y)=0.5+(\sin ^2(x^2-y^2)-0.5)/(1+0.001(x^2+y^2))^2.
\end{eqnarray*}
The  global minimum is $(0,0)$, with function value $0$. The initial point (randomly chosen) is $(-57.32135254,-17.85920667)$, at which the function value is 0.514. In this case, New Q-Newton's method diverges to infinity.

\begin{table}[htp]
\fontsize{11}{11}\selectfont
  \centering
  \begin{tabular}{|l|c|c|c|c|c|c|c|c|}
  \hline
&ACR&BFGS	& Newton   & NewQ  & V1&V2& Iner& Back\\
\hline
&\multicolumn{8}{c|}{f1}\\
\hline
Iterations &Error&34& 1e+4 &1e+4  &35&18&1e+4 & 11\\
\hline
$f$ &Error&2e-41&[nan] &[nan]  &8e-15 &2e-15 & 5e+15 &2e-15 \\
\hline
$||\nabla f||$ &Error&9e-11  &[nan]  &[nan] &0.004&0.003 &1e+4 &0.003\\
\hline
Time &Error&0.003&0.113 & 0.140&0.001 &0.006 &0.132 &0.005\\
\hline
&\multicolumn{8}{c|}{f2}\\
\hline
Iterations &3&4&6 &6 &8&5& 1948&13 \\
\hline
$f$ &-0.0118&-0.0118&-2e-7&-2e-7 &-0.0118 &-0.0118 & -0.0118 &-0.0118 \\
\hline
$||\nabla f||$ &9e-5&4e-10  &6e-17 &6e-17&8e-17& 5e-17& 2e-8&9e-10\\
\hline
Time &0.007&0.003&0.0002 & 0.0002&0.0004 &0.0002 &0.034 &0.016\\
\hline
&\multicolumn{8}{c|}{f3}\\
\hline
Iterations &0&5&Error & 1e+4&20&20& 457&172 \\
\hline
$f$ &102&0.685&Error& 0.501&0.006 &0.002 & 9e+65 & 0.0005\\
\hline
$||\nabla f||$ &282&0.707  &Error&1.577&0.707& 1.342&1 &0.712\\
\hline
Time &0.001&0.007&Error &2.549 & 0.008&0.008 &0.011&0.028\\
\hline
&\multicolumn{8}{c|}{f4}\\
\hline
Iterations &1e+4&17&9 &23 &24&23&16 &55 \\
\hline
$f$ &0.137&5e-13&2.572& 4.076&6e-11 & 4e-11&20.936  & 6e-11\\
\hline
$||\nabla f||$ &3.855&2e-11 &3e-12&3e-13&7e-10& 5e-10& 0&8e-10\\
\hline
Time &51.61&0.195& 0.167&0.431 &0.443 & 0.436&0.072&0.924\\
\hline
&\multicolumn{8}{c|}{f5}\\
\hline
Iterations &5&13&9 &9 &6&7& 14&12 \\
\hline
$f$ &46.762&46.762&138&58.702& 43.777&46.762 &2e+33  &46.762 \\
\hline
$||\nabla f||$ &5e-6&1e-11 &2e-11&7e-12&5e-9& 1e-8&0 &9e-7\\
\hline
Time &0.018&0.105&0.170 &0.183 &0.120 &0.138 &0.063&0.246\\
\hline
&\multicolumn{8}{c|}{f6}\\
\hline
Iterations &10&19&21 &172 &12&16&Error & 187\\
\hline
$f$ &4e-11&6e-24&14.203&7.312&0 & 0& Error&2e-18\\
\hline
$||\nabla f||$ &2e-5&1e-11&0&4e-16&6e-22& 6e-22&Error&1e-9\\
\hline
Time &0.023&0.073&0.149 &1.231&0.088& 0.118&Error&1.322\\
\hline

 \end{tabular}
   \caption{Performance of different optimization methods for some benchmark functions, with random initial points. For descriptions of the functions and the initial points, see Section \ref{Section:BenchmarkFunctions}.  See also Table \ref{tab:BenchmarkFunctions2}.}%  
  \label{tab:BenchmarkFunctions}
\end{table}

\begin{table}[htp]
\fontsize{11}{11}\selectfont
  \centering
  \begin{tabular}{|l|c|c|c|c|c|c|c|c|}
  \hline
&ACR&BFGS	& Newton   & NewQ  & V1&V2& Iner& Back\\
\hline
&\multicolumn{8}{c|}{f7, Point 1}\\
\hline
Iterations &28&1&Error &1e+4 &45&10&14 & 6\\
\hline
$f$ &141.929&3.029&Error&191& 0.098&1.271 &3e+17 &2.413\\
\hline
$||\nabla f||$ &35.310&140&Error&26.092&0.008&1.617 &0&2.276\\
\hline
Time &0.022&0.142&Error &67.578&0.315&0.071 &0.024&0.062\\
\hline
&\multicolumn{8}{c|}{f7, Point 2}\\
\hline
Iterations &1e+4&3&Error &1e+4 &11&5&14 & 39\\
\hline
$f$ &28.247&1.038&Error&49.176&0.004 &3.378&8e+17 &0.003\\
\hline
$||\nabla f||$ &180&0.791&Error&101&0.005&1.777 &0&0.005\\
\hline
Time &15.578&0.138&Error &67.778&0.079&0.035&0.025&0.273\\
\hline
&\multicolumn{8}{c|}{f8}\\
\hline
Iterations &1124&6&25 &17 &1e+4&1e+4& 3&1e+4 \\
\hline
$f$ &0.001&0.475&0.5&0.500 &0.499 &0.476&0.5&0.476 \\
\hline
$||\nabla f||$ &0.002&0.177  & 6e-10&1e-11&1e-10& 7e-4& 0&0.001\\
\hline
Time &5.550&0.132&0.175 & 0.123&74.551 &74.091 & 0.008&69.433\\
\hline

 \end{tabular}
   \caption{Performance of different optimization methods for some benchmark functions, with random initial points. For descriptions of the functions and the initial points, see Section \ref{Section:BenchmarkFunctions}. See also Table \ref{tab:BenchmarkFunctions}.}%  
  \label{tab:BenchmarkFunctions2}
\end{table}

\subsubsection{Behaviour near degenerate saddle points}\label{SectionSaddlePoints} Here, we explore the behaviour of different optimization methods near some typical {\bf degenerate} saddle points. Since in this case the numerical values can become very large or very small in just a few iterates, we restrict the maximum number of iterations to 50.  See Table \ref{tab:SaddlePoints}. 

We recall that a degenerate saddle point is a critical point, which is neither a local minimum nor a local maximum, at which the Hessian is not invertible. Indeed, for all experiments to be considered, the Hessian is 0 identically at the concerned degenerate saddle point. 

We consider the "monkey saddle": 

\begin{eqnarray*}
f_9(x,y)=x^3-3xy^2. 
\end{eqnarray*}
The point $(0,0)$ is a degenerate saddle point. The (randomly chosen) initial point near the saddle point is $(-0.0004322,   0.00093845)$, at which the function value is $1e-9$. 

We consider the function: 
\begin{eqnarray*}
f_{10}(x,y)=x^2y+y^2. 
\end{eqnarray*}
The point $(0,0)$ is a degenerate saddle point. The (randomly chosen) initial point near the saddle point is $(0.0007154,  0.00088668)$, at which the function value is $7e-7$. 

We consider a function of the type $\sum _{i,j=1}^Dq_{i,j}x_i^2x_j^2$, suggested in \cite{lee-simchowitz-jordan-recht} as a good class of test functions for algorithms aiming to avoid (degenerate) saddle points. It is simple to check that if $q_{i,i}<0$ for some i and $q_{j,j}>0$ for some $j$, then the origin $(0,\ldots ,0)$ is a degenerate saddle point. Hence, a randomly chosen such matrix will have a degenerate saddle point at the origin. For certainty, we choose $D=3$ and choose $\{q_{i,j}\}$ (randomly) as: 
 \[ \left( \begin{array}{ccc}
-6.53899332&-4.918748445&-1.884110645\\
-4.918748445&-8.26397796&2.280742435\\
-1.884110645&2.280742435&1.36728532\\
\end{array}\right) \]

The (randomly chosen) initial point is $(8.52766549e-05, -4.64890817e-04,  2.75958449e-04)$, for which the function value is $-3e-13$. 
 
We consider a parametrised version of the function $f_{10}:$
\begin{eqnarray*}
f_{12}(x,y,t)=(x^2y+y^2)t.
\end{eqnarray*}
The whole line $\{x=y=0\}$ consists of degenerate saddle points.  The (randomly chosen) initial point near the origin is $ (0.00040449,  0.00029101, -0.00029746)$, at which the function value is $-2e-11$.

\begin{table}[htp]
\fontsize{11}{11}\selectfont
  \centering
  \begin{tabular}{|l|c|c|c|c|c|c|c|c|}
  \hline
&ACR&BFGS	& Newton   & NewQ  & V1&V2& Iner& Back\\
\hline
&\multicolumn{8}{c|}{f9}\\
\hline
Iter. &2&36&24 &50&50&50& 50&34 \\
\hline
$f$ &-2e+7&-536.40&2e-31&-2e+15& -1e+4& -1e+4&-3e+147 &-1.5e+193\\
\hline
$||\nabla f||$ &2e+5&1e+5&1e-20&5e+10&1e+3& 1e+3&0&0\\
\hline
Time &0.013&0.479&0.189 &0.384&0.390&0.402 &0.141&0.335\\
\hline
&\multicolumn{8}{c|}{f10}\\
\hline
Iter. &4&2&32 &50 &50&50&50 &50 \\
\hline
$f$ &-2e+34&-6e+53&-2e-40&-1e+12& -6e+3&-6e+3&-2e+119 &-6e-14\\
\hline
$||\nabla f||$ &2e+23&0&8e-31&3e+8&741&741 &0&4e-10\\
\hline
Time &0.022&0.159&0.244 &0.387&0.386&0.391&0.141&0.491\\
\hline
&\multicolumn{8}{c|}{f11}\\
\hline
Iter. &3&1& 50&50 &50&50&50 & 23\\
\hline
$f$ &-7e+12&-5e+63&-1e-36&-3e+12 &-3e+5 &-3e+5&3e+100&-1e+70 \\
\hline
$||\nabla f||$ &3e+10&0  & 1e-26&1e+10&1e+5&1e+5 &0 &0\\
\hline
Time &1.328&0.171&0.632 &0.646 &0.646 &0.638 &0.186 &0.356\\
\hline
&\multicolumn{8}{c|}{f12}\\
\hline
Iter. &3&1&50 &50 &50&50&50 & 10\\
\hline
$f$ &1e+30&-8e+69&3e-48&-6e+15 &-5329 &-5329&-9e+158&-1e+139 \\
\hline
$||\nabla f||$ &1e+28& 0& 2e-31&7e+10&665& 665&0 &0\\
\hline
Time &0.018&0.202& 0.722&0.742 &0.730 &0.735 &0.182 &0.147\\
\hline

 \end{tabular}
   \caption{Performance of different optimization methods, with random initial points near degenerate saddle points. For descriptions of the functions and the initial points, see Section \ref{SectionSaddlePoints}. "Iter" means the number of iterations. In all examples, Newton's method is attracted to the saddle point. In all examples, Unbounded Two-way Backtracking GD escapes from the saddle point, even though for f10 it does so very slowly. For f12: While theoretically ACR has descent property, its implementation which is not the same as the theoretical version, does not have this property.}%  
  \label{tab:SaddlePoints}
\end{table}

\section{Conclusions and Future outlook}

In this paper, based on the algorithm New Q-Newton's method in our previous joint work \cite{truong-etal}, we  define a new algorithm New Q-Newton's method Backtracking. By incorporating Armijo's condition into New Q-Newton's method, it achieves as far as we know the best theoretical guarantee among all iterative optimisation methods for Morse functions: with a random initial point, it assures convergence to a {\bf local minimum} and with a {\bf quadratic} rate of convergence. An interesting and useful open question is to explore the behaviour of New Q-Newton's method Backtracking near degenerate critical points. See some experiments in Section \ref{SectionSaddlePoints}. 

Some variants of New Q-Newton's method Backtracking are also introduced, and prove some similar theoretical results. All of them are descriptive enough to be straight forwardly implemented in Python. Experiments show that New Q-Newton's method Backtracking indeed improves the performance of New Q-Newton's method, in many cases significantly. 

Using ideas from \cite{truongnew}, New Q-Newton's method Backtracking can be extended to Riemannian manifolds settings. This can be used appropriately for constrained optimization. For example, experiments in Section 3 show that New Q-Newton's method Backtracking and variants work well for some non-smooth functions on $\mathbb{R}^m$, and hence it is interesting to explore its theoretical guarantees for functions which are not $C^3$ or even not differentiable. If the non-$C^3$ set $\mathcal{C}$ is small (for example, has Lebesgue measure 0), one can work with the Riemannian version of New Q-Newton's method Backtracking and variants on the complement $U=\mathbb{R}^m\backslash \mathcal{C}$, regarded as a Riemannian manifold.

Thus, both theoretically and experimentally, New Q-Newton's method Backtracking combines the best of two worlds: it has the good convergence guarantee of Backtracking line search, while has the fast convergence of Newton's method. 

Experiments show that usually the weaker descent property $f(x_{n+1})-f(x_n)\leq 0$ is faster than Armijo's condition to combine with New Q-Newton's method, while the two variants generally converge to the same points. The reason can be that both Armijo's condition and New Q-Newton's method have their own pros and cons, and sometimes it can be overdo when combining them together. In such cases, the weaker descent property may be "just enough" to prevent New Q-Newton's method from diverging to infinity, while Armijo's method is "overdo".  (On the other hand, experiments in both small scale and Deep Neural Networks show that the weaker descent property combined with the usual Gradient Descent method can perform very much poorly than Backtracking GD.)

We also find that while using the rescaled $\widehat{w_k}=\frac{w_k}{\max \{1,||w_k||\}}$, which is a bounded vector, allows one to prove slightly better theoretical guarantees, there are cases where using the original step direction $w_k$ of New Q-Newton's method in Backtracking line search performs much better in practice. The reason could be the same as in the previous paragraph, that is $w_k$ is generally already good, and so one should not reduce its size too much. 

In the stochastic Griewank function experiment, it is interesting to note that it seems that methods employed Backtracking line search (New Q-Newton's method Backtracking  V1--4 and Unbounded Two-way Backtracking GD) behave better when the variance $\sigma$ increases. Hence, they could be suitably used for real life data, when one has no control on the variance.  Newton's method and New Q-Newton's method in contrast behave worst when $\sigma$ increases. 

Hence New Q-Newton's method Backtracking is a good second order method candidate to use in large scale optimization such as Deep Neural Networks. While Backtracking line search is not too time consuming (in particular if one uses the Two-way Backtracking version, see \cite{truong-nguyen1, truong-nguyen2}), current techniques for computing the Hessian of the function and its eigenvectors and eigenvalues (approximately) is both time consuming and resource demanding and hence is an obstruction to effectively use New Q-Newton's method (Backtracking) and variants in large scale. Research on this front is very demanding in both personnel and computing resources, but is worthy to pursue given the good theoretical guarantees and the observed experiments in small scale, and hence awaits to be explored in the future.

\end{document}